\documentclass[12pt,a4paper]{amsart}
\usepackage{a4wide,fullpage,setspace}
\usepackage{amssymb,amsmath,mathtools}
\usepackage{tikz,tikz-cd}
\usetikzlibrary{arrows.meta,decorations.pathreplacing}
\usepackage[shortlabels]{enumitem}
\usepackage[numbers,sort&compress]{natbib}
\definecolor{dark-red}{rgb}{0.5,0.15,0.15}
\usepackage[colorlinks=true,linkcolor=black,citecolor=dark-red,urlcolor=dark-red]{hyperref}
\swapnumbers

\makeatletter
\let\leq\@undefined
\let\geq\@undefined
\let\vec\@undefined
\let\phi\@undefined
\let\epsilon\@undefined
\let\injlim\@undefined
\let\projlim\@undefined
\makeatother
\newcommand{\leq}{\leqslant}
\newcommand{\geq}{\geqslant}
\newcommand{\vec}{\overrightarrow}
\newcommand{\phi}{\varphi}
\newcommand{\epsilon}{\varepsilon}
\newcommand{\injlim}{\varinjlim}
\newcommand{\projlim}{\varprojlim}
\newcommand{\Set}{\mathsf{Set}}
\newcommand{\Flow}{\mathsf{Flow}}
\newcommand{\FLOW}{\mathsf{FLOW}}
\newcommand{\Top}{\mathsf{Top}}
\newcommand{\TOP}{\mathsf{TOP}}
\newcommand{\PP}{\mathbb P}

\newcommand{\Ch}{\operatorname{Ch}}

\newcommand{\Glob}{\operatorname{Glob}}

\newcommand{\Poset}{\mathrm{PoSet}}
\newcommand{\Sp}{\mathsf{S}}
\newcommand{\Di}{\mathsf{D}}
\newcommand{\cocartesian}{\arrow[lu, phantom, "\ulcorner"{font=\Large}, pos=0]}

\newtheorem*{thmN}{Theorem}
\newtheorem*{corN}{Corollary}
\newtheorem{thm}{Theorem}[section]
\newtheorem{proposition}[thm]{Proposition}
\newtheorem{lem}[thm]{Lemma}
\newtheorem{cor}[thm]{Corollary}
\theoremstyle{definition}
\newtheorem{definition}[thm]{Definition}

\title{The chain replacement of a poset flow}
\author[P. Gaucher]{Philippe Gaucher}
\address{Universit\'e Paris Cit\'e, CNRS, IRIF, F-75013, Paris, France}
\urladdr{\url{https://www.irif.fr/~gaucher}}

\begin{document}

\begin{abstract}
We introduce the chain replacement of a poset flow: it is obtained by considering the simplicial nerves of the posets of strictly increasing chains in the given poset, ordered by refinement. It maps posets to q-cofibrant flows and inclusions of posets to q-cofibrations. Using the combinatorial properties of the chain replacement, we prove that pushouts along the chain replacement of an order-reflecting inclusion of posets induce trivial h-cofibrations on spaces of execution paths between states of the original flow. By introducing the Hurewicz model structure on flows (or H-model structure), we deduce the same property for any q-cofibrant replacement of an order-reflecting inclusion of posets.
\end{abstract}

\maketitle
\setcounter{tocdepth}{1}
\tableofcontents
\hypersetup{linkcolor = dark-red}

\section{Introduction}

\subsection*{Presentation} Flows provide a convenient topological model for directed phenomena arising in concurrency \cite{DAT_book}: they are small semicategories enriched over a convenient category of topological spaces, whose morphism spaces are spaces of execution paths. In this framework, posets are regarded as flows, with one execution path from \(x\) to \(y\) precisely when \(x<y\). Such poset flows occur naturally in the homotopy theory of flows, for instance in the definition of dihomotopy \cite{hocont} and in the study of globular subdivisions \cite{MultipointedSubdivision}. Poset flows are not q-cofibrant in general. We introduce an explicit q-cofibrant replacement, the \emph{chain replacement}. For a poset \(P\), the flow \(\Ch(P)\) has the same states as \(P\), and its space of execution paths from \(x\) to \(y\) is the geometric realization of the nerve of the poset of strict chains from \(x\) to \(y\), ordered by refinement. This construction is functorial in strictly increasing maps, takes posets to q-cofibrant flows, and sends inclusions of posets to q-cofibrations.

The main result concerns the effect of attaching such a replacement to an arbitrary flow. We prove that:
\begin{samepage}
\begin{thmN}[Theorem~\ref{thm:r}]
Let \(i: P\subset Q\) be an order-reflecting inclusion of posets, and let \(i^{\mathrm{cof}}: P^{\mathrm{cof}}\to Q^{\mathrm{cof}}\) be any \(q\)-cofibrant replacement of \(i\). In every pushout diagram
\[
	\begin{tikzcd}
		P^{\mathrm{cof}} \arrow[r] \arrow[d,"i^{\mathrm{cof}}"',rightarrowtail] & X \arrow[d,"f",rightarrowtail] \\
		Q^{\mathrm{cof}} \arrow[r] & Y \cocartesian,
	\end{tikzcd}
	\]
the induced map
\[
	\PP_{u,v}X\longrightarrow \PP_{f(u),f(v)}Y
	\]
is a trivial h-cofibration of spaces for all states \(u,v\in X^{0}\).
\end{thmN}
\end{samepage}
We recover \cite[Proposition~8.8]{MultipointedSubdivision} as a particular case. This result is used in \cite{MultipointedSubdivision} to prove that globular subdivisions are dihomotopy equivalences in the sense of \cite{hocont}. Thus the homotopy type of the space of execution paths between any two states of \(X\) is preserved. The order-reflecting hypothesis is essential: if \(x\not<y\) in \(P\) but \(x<y\) in \(Q\), then \(\PP_{x,y}\Ch(P)\) is empty and \(\PP_{x,y}\Ch(Q)\) is nonempty (Figure~\ref{fig:order-reflecting}).

We first prove the result for the chain replacement (Theorem~\ref{thm:canonical}). The simplices in the nerves of the posets of chains give a factorization
\[
\begin{tikzcd}[column sep=large]
\Ch(P) \arrow[r,rightarrowtail,"\simeq"] & \Ch(Q)|_P \arrow[r,hook] & \Ch(Q).
\end{tikzcd}
\]
The first map is a trivial q-cofibration when \(P\subset Q\) is order-reflecting (Theorem~\ref{thm:order-reflecting-effect}). A pushout along the second map preserves the spaces of execution paths between the states of the original flow (Proposition~\ref{prop:full-subflow-pushout}). To prove the result for an arbitrary q-cofibrant replacement, we use the Hurewicz model structure defined by the topological enrichment. All flows are H-cofibrant, every q-cofibration is an H-cofibration, and weak equivalences of the q-model structure between q-cofibrant flows are weak equivalences of the H-model structure. The gluing lemma then gives the comparison with any q-cofibrant replacement of \(P\subset Q\).

We deduce by a standard transfinite induction:
\begin{samepage}
\begin{corN} [Corollary~\ref{cor:final}]
Let \(f:X\to Y\) be a retract of a transfinite composition of pushouts of q-cofibrant replacements of order-reflecting inclusions of posets. Then, for any pair of states \(u,v\in X^0\), \(f\) induces a trivial h-cofibration of spaces
\[
	\begin{tikzcd}
		\PP_{u,v} X \arrow[r,"\simeq"] & \PP_{f(u),f(v)} Y.
	\end{tikzcd}
	\]
In particular, it is a homotopy equivalence.
\end{corN}
\end{samepage}

\subsection*{Future work} The combinatorial constructions developed in this paper provide the main technical tool for a subsequent paper, where we prove the cubical analogue of \cite{MultipointedSubdivision}, namely that barycentric subdivisions are dihomotopy equivalences in the sense of \cite{hocont}.

\subsection*{Outline of the paper} Section~\ref{sec:notations} fixes the conventions on flows, posets, and order-reflecting inclusions. Section~\ref{sec:chain-replacement} defines the chain replacement and proves its basic cellular properties. Section~\ref{sec:subflow-lemma} proves the full subflow pushout lemma. Section~\ref{sec:pushout-along-chain-replacement} proves the main result for the chain replacement. Section~\ref{sec:H-model-structure} constructs the Hurewicz model structure on flows. Section~\ref{sec:comparison-q-H} compares this model structure with the q-model structure. Finally, Section~\ref{sec:pushout-along-q-replacement} combines these ingredients to prove the theorem for arbitrary q-cofibrant replacements.

\section{Notation and order-reflecting inclusions}
\label{sec:notations}

\(\Top\) denotes either the category of \(\Delta\)-generated spaces or the category of \(\Delta\)-Hausdorff \(\Delta\)-generated spaces (cf. \cite[Section~2 and Appendix~B]{leftproperflow}). It can be equipped with its q-model structure or its h-model structure. \(\Top\) is locally presentable \cite[Corollary~3.7]{FR} and cartesian closed \cite[Proposition~2.5]{mdtop}; its internal hom is denoted by \(\TOP(-,-)\). By compact, we mean quasi-compact Hausdorff.

A \textit{flow} is a small semicategory enriched over \(\Top\). The set of objects of a flow \(X\), called its \textit{states}, is denoted by \(X^0\). The space of morphisms, called \textit{execution paths}, from \(u\) to \(v\) is denoted by \(\PP_{u,v}X\). The strictly associative composition law is denoted by \(*\). The category of flows is denoted by \(\Flow\). Every set can be viewed as a flow without execution paths. The flow \(\Glob(Z)\) is defined by
\[
\Glob(Z)^0=\{0,1\},\qquad \PP\Glob(Z)=\PP_{0,1}\Glob(Z)=Z.
\]
\(\Flow\) can be equipped with its q-model structure (e.g. \cite[Theorem~7.4]{QHMmodel}). The q-model structure is cofibrantly generated. The set of generating q-cofibrations is \(I^{\mathrm{gl}} \cup \{C:\varnothing\to \{0\},R:\{0,1\}\to \{0\}\}\) with
\[
I^{\mathrm{gl}}=\{\Glob(\Sp^{n-1})\subset \Glob(\Di^n) \mid n\geq 0\}
\]
where \(\Sp^{n-1}=\{(x_1,\dots,x_n)\in \mathbb R^n\mid \sum_i x_i^2=1\}\) is the \((n-1)\)-dimensional sphere, \(\Sp^{-1}=\varnothing\) and \(\Di^{n}=\{(x_1,\dots,x_n)\in \mathbb R^n\mid \sum_i x_i^2\leq 1\}\) is the \(n\)-dimensional disk. The set of generating trivial q-cofibrations is
\[
J^{\mathrm{gl}}=\{\Glob(\Di^{n}\times\{0\})\subset \Glob(\Di^n\times[0,1]) \mid n\geq 0\}.
\]
A \textit{relative cell complex} is a transfinite composition of pushouts of maps in \(I^{\mathrm{gl}}\cup\{C,R\}\). A \textit{cell complex} \(X\) is a flow \(X\) such that the canonical map \(\varnothing\to X\) is a relative cell complex.

If \(B\) is a flow whose states are identified with \(Q\), we write \(B|_P\) for the \textit{full} subflow on the states of \(P\):
\[
(B|_P)^0=P,
\qquad
\PP_{x,y}(B|_P)=\PP_{x,y}B
\quad (x,y\in P).
\]

A (small) poset \(R\) is viewed as the flow with set of states \(R\), with one execution path from \(x\) to \(y\) if \(x<y\), and no execution path otherwise.

Let \(\Poset^+\) denote the category of posets and strictly increasing maps. An inclusion of posets \(P\subset Q\) is called \emph{order-reflecting} if for all \(x,y\in P\),
\[
x\leq_Q y \quad\Longrightarrow\quad x\leq_P y .
\]
Equivalently, the order of \(P\) is the order induced from \(Q\). Every order-reflecting inclusion of posets is a morphism of \(\Poset^+\).

For a small category \(\mathcal C\), \(N(\mathcal C)\) denotes the simplicial nerve of \(\mathcal C\) and \(|N(\mathcal C)|\) denotes its geometric realization. The geometric realization of a simplicial set is obtained by attaching an \(n\)-dimensional cell for each nondegenerate \(n\)-simplex of the given simplicial set. By \cite[Proposition~B.15 and B.16]{leftproperflow}, the topology of the geometric realization is therefore the final topology for the cells. This fact is used in the proofs of Proposition~\ref{prop:chain-replacement} and Proposition~\ref{prop:restriction-cofibration} to prove that some continuous bijections are homeomorphisms. By \cite[Lemma~3.1.8]{MR99h:55031}, the comparison map
\[
|\Delta^m\times \Delta^n|\longrightarrow |\Delta^m|\times |\Delta^n|
\]
is a homeomorphism which is natural in \(m\) and \(n\), where \(\Delta^k\) denotes the representable \(k\)-dimensional simplex. Thus the realization functor in any cartesian closed category of topological spaces preserves finite products. There are natural homeomorphisms
\[
|N(\mathcal C\times \mathcal D)| \cong |N(\mathcal C) \times N(\mathcal D)| \cong |N(\mathcal C)| \times |N(\mathcal D)|
\]
since the simplicial nerve functor is a right adjoint \cite[Theorem~14.1.5]{ref_model2}.

\section{The chain replacement of a poset}
\label{sec:chain-replacement}

Let \(P\) be a poset.  For \(x<y\) in \(P\), let \(\mathcal C_P(x,y)\) be the poset of strict chains
\[
\gamma=(x=x_0<x_1<\cdots <x_k=y)
\]
ordered by refinement: \(\gamma\leq \gamma'\) if \(\gamma\) is obtained from \(\gamma'\) by removing intermediate vertices.  Thus the chain \((x<y)\) is an initial object of \(\mathcal C_P(x,y)\).  If \(x\not<y\), put \(\mathcal C_P(x,y)=\varnothing\). If \(x<y\) in \(P\), and if
\[
\sigma=(\gamma_0\leq \gamma_1\leq \cdots\leq \gamma_m)
\]
is an \(m\)-simplex of \(N\mathcal C_P(x,y)\), write \(V(\gamma_i)\) for the set of vertices of the strict chain \(\gamma_i\), and put
\[
V_P(\sigma)=\bigcap_{i=0}^{m}V(\gamma_i) = V(\gamma_0).
\]
Thus \(V_P(\sigma)\) is the set of vertices of \(P\) through which every chain \(\gamma_i\) passes. It contains \(x\) and \(y\). We say that \(\sigma\) is indecomposable if
\[
V_P(\sigma)=\{x,y\}.
\]
Equivalently, there is no intermediate element \(z\in P\), \(x<z<y\), through which all the chains \(\gamma_i\) factor. Note that this does not mean that the open interval \(]x,y[\) is empty. We call \(\sigma\) \emph{decomposable} if it is not indecomposable.

For each \(m\)-simplex \(\sigma\) of \(N\mathcal C_P(x,y)\), write \(\tau_\sigma\) for its unique nondegenerate reduction given by the Eilenberg--Zilber lemma:
\[
\sigma=(s_\sigma)^*(\tau_\sigma),\quad \tau_\sigma\hbox{ nondegenerate,}\quad s_\sigma:\Delta^m\to\Delta^{n_\sigma}, \quad m\geq n_\sigma.
\]
The pair \((s_\sigma,\tau_\sigma)\) is unique by \cite[Chapter~II, \S3, Proposition~1, pp.~26--27]{GabrielZisman1967}.

Define a flow \(\Ch(P)\) as follows.  Its set of states is \(P\), and
\[
\PP_{x,y}\Ch(P)=
\begin{cases}
|N\mathcal C_P(x,y)|,& x<y,\\
\varnothing,& x\not<y.
\end{cases}
\]
Composition is induced by concatenation of strict chains:
\[
\mathcal C_P(x,y)\times \mathcal C_P(y,z)
\longrightarrow
\mathcal C_P(x,z),
\qquad
(\gamma,\delta)\longmapsto \gamma*\delta.
\]
It is strictly associative.  There is an augmentation
\[
\epsilon_P:\Ch(P)\longrightarrow P
\]
which is the identity on states and collapses each nonempty path space to the unique path of the poset flow \(P\).

Since a map of \(\Poset^+\) sends a strict chain to a strict chain, we obtain the following definition:

\begin{samepage}
\begin{definition}
The assignment
\[
	P\mapsto \Ch(P)
	\]
defines a functor \(\Ch:\Poset^+\to\Flow\) called the \textit{chain replacement functor}.
\end{definition}
\end{samepage}

\begin{samepage}
\begin{lem}[Normal form] \label{lem:normal-form}
Let \(P\) be a poset and let
\[
	\sigma\in N\mathcal C_P(x,y)_m.
	\]
There exist unique elements
\[
	x=z_0<z_1<\cdots<z_r=y
	\]
and a unique decomposition
\[
	\sigma=\sigma_1*\cdots *\sigma_r
	\]
such that
\[
	\sigma_j\in N\mathcal C_P(z_{j-1},z_j)_m
	\]
is indecomposable for every \(1\leq j\leq r\). Each \(\sigma_j\) has a unique Eilenberg--Zilber decomposition
\[
	\sigma_j=(s_j)^*(\tau_j),
	\]
where \(\tau_j\) is nondegenerate and indecomposable. Consequently, every simplex \(\sigma\) has a unique normal form
\[
	\sigma
	=
	(s_1)^*(\tau_1)*\cdots*(s_r)^*(\tau_r)
	\]
as a composite of degeneracies of nondegenerate indecomposable simplices.
\end{lem}
\end{samepage}

\begin{proof}
The elements \(z_0,\ldots,z_r\) are precisely the elements of
\[
	V_P(\sigma)=\bigcap_{i=0}^m V(\gamma_i)=V(\gamma_0),
	\]
written in increasing order. Every chain in \(\sigma\) contains each \(z_j\), so restriction to the intervals \([z_{j-1},z_j]\) yields the decomposition
\[
	\sigma
	=
	\sigma|_{[z_0,z_1]}*\cdots*
	\sigma|_{[z_{r-1},z_r]}.
	\]
These restrictions are indecomposable, and any decomposition into indecomposable simplices is determined by the same common vertices. The decomposition is therefore unique; see Figure~\ref{fig:normal-form}.

The Eilenberg--Zilber lemma gives, for each indecomposable simplex \(\sigma_j\), a unique pair \((s_j,\tau_j)\) such that \(\sigma_j=(s_j)^*(\tau_j)\) and \(\tau_j\) is nondegenerate. Moreover, \(\tau_j\) is indecomposable. This proves the uniqueness of the normal form.
\end{proof}

\begin{figure}[htbp]
\centering
\begin{tikzpicture}[x=1.45cm,y=1.05cm,>=Stealth]
\node[left] at (-.4,1) {$\gamma_0$};
\node[left] at (-.4,0) {$\gamma_1$};
\foreach \xx/\lab in {0/x,2/z,4/y} {
  \fill (\xx,1) circle (2pt) node[above=3pt] {$\lab$};
  \fill (\xx,0) circle (2pt) node[below=3pt] {$\lab$};
}
\foreach \xx/\lab in {1/a,3/b} {
  \draw (\xx,0) circle (2pt) node[below=3pt] {$\lab$};
}
\draw[->,shorten >=4pt,shorten <=4pt] (0,1)--(2,1);
\draw[->,shorten >=4pt,shorten <=4pt] (2,1)--(4,1);
\foreach \xx in {0,1,2,3} {
  \draw[->,shorten >=4pt,shorten <=4pt] (\xx,0)--(\xx+1,0);
}
\draw[densely dashed] (2,.12)--(2,.88);
\path (2,1.5);
\draw[decorate,decoration={brace,mirror,amplitude=4pt}] (0,-.65)--(1.95,-.65) node[midway,below=7pt] {$\sigma|_{[x,z]}$};
\draw[decorate,decoration={brace,mirror,amplitude=4pt}] (2.05,-.65)--(4,-.65) node[midway,below=7pt] {$\sigma|_{[z,y]}$};
\end{tikzpicture}
\caption{For \(P=\{x<a<z<b<y\}\) and \(\sigma=(\gamma_0<\gamma_1)\), the common vertices are \(V_P(\sigma)=\{x,z,y\}\). The two restrictions are indecomposable and \(\sigma=\sigma|_{[x,z]}*\sigma|_{[z,y]}\).}
\label{fig:normal-form}
\end{figure}

\begin{samepage}
\begin{proposition}\label{prop:chain-replacement}
For every poset \(P\), the flow \(\Ch(P)\) is q-cofibrant and \(\epsilon_P:\Ch(P)\to P\) is a trivial q-fibration of flows.  Moreover, if \(P\subset Q\) is an inclusion of posets, the induced map \(\Ch(P)\to \Ch(Q)\) is a q-cofibration. More precisely, it is a relative cell complex.
\end{proposition}
\end{samepage}

\begin{proof}
We construct \(\Ch(P)\) as a cell complex.

Starting from the discrete flow \(G^{-1}(P)\) with set of states \(P\), attach, in increasing dimension, for all \(x,y\in P\), one globular \(m\)-dimensional cell from \(x\) to \(y\)
\[
	\Glob(|\partial\Delta^m|)
	\longrightarrow
	\Glob(|\Delta^m|),
	\]
which is a generating q-cofibration of flows up to the usual identifications \(|\partial\Delta^m|\cong \Sp^{m-1}\) and \(|\Delta^m|\cong \Di^m\), for every \textit{indecomposable} nondegenerate \(m\)-simplex \(\sigma\) of \(N\mathcal C_P(x,y)\), with the convention \(|\partial\Delta^0|=\varnothing\):
\[
	\begin{tikzcd}
		\displaystyle\coprod_{\sigma} \Glob(|\partial\Delta^m|)\arrow[r]\arrow[d] & G^{m-1}(P) \arrow[d]\arrow[r]&\Ch(P) \arrow[d,equal]\\
		\displaystyle\coprod_{\sigma} \Glob(|\Delta^m|) \arrow[r,"\coprod_{\sigma} \chi_\sigma"]\arrow[rr,bend right=2em,"\coprod_\sigma \Psi_\sigma"' pos=0.45] & G^{m}(P)\cocartesian \arrow[r,dashed]&\Ch(P),
	\end{tikzcd}
	\]
where, for each indecomposable nondegenerate \(m\)-simplex \(\sigma\) of \(N\mathcal C_P(x,y)\), \(|\sigma|:|\Delta^m|\to |N\mathcal C_P(x,y)|\) induces the map \(\Psi_\sigma:\Glob(|\Delta^m|) \to \Ch(P)\) such that \(|\sigma|=\PP_{0,1}\Psi_\sigma\). The universal property of the pushout gives the dashed arrow.

We check that the attaching maps are well defined for \(m \geq 0\). We first introduce some notations. For a general indecomposable \(m\)-simplex \(\sigma\), write
\[
	\begin{tikzcd}[column sep=4em]
		\chi_\sigma : \Glob(|\Delta^m|)\arrow[r,"\Glob(|s_\sigma|)"]& \Glob(|\Delta^{n_\sigma}|) \arrow[r,"\chi_{\tau_\sigma}"] & G^{n_\sigma}(P)
	\end{tikzcd}
	\]
where \((s_\sigma,\tau_\sigma)\) is as above. For a general \(m\)-simplex \(\sigma\), if \(V_P(\sigma)=\{x=z_0<z_1<\cdots<z_r=y\}\) for \(r\geq 1\), then every chain in \(\sigma\) contains every \(z_j\). Each restriction \(\sigma|_{[z_{j-1},z_j]}\) is indecomposable and \(\sigma\) is therefore the unique composite of indecomposable simplices: \( \sigma = \sigma|_{[z_{0},z_1]}*\cdots *\sigma|_{[z_{r-1},z_r]} \). Write
\[
	\chi_\sigma = \chi_{\sigma|_{[z_{0},z_1]}}*\cdots *\chi_{\sigma|_{[z_{r-1},z_r]}}.
	\]
Let \(\sigma=(\gamma_0,\cdots,\gamma_m)\) be an indecomposable nondegenerate \(m\)-simplex of \(N\mathcal C_P(x,y)\), and let \(\partial_i\sigma\) be one of its faces of codimension one. If \(i>0\), then \(\partial_i\sigma\) is indecomposable, since \(\gamma_0\) remains one of its vertices. This implies that the characteristic map \(\chi_{\partial_i\sigma}\) is already defined. Otherwise, write
\[
	V_P(\partial_0\sigma)=\{x=z_0<z_1<\cdots<z_r=y\}.
	\]
Every chain in \(\partial_0\sigma\) contains every \(z_j\). Therefore \(\partial_0\sigma\) is the unique composite of its restrictions
\[
	(\partial_0\sigma)|_{[z_{j-1},z_j]},
	\qquad (1\leq j\leq r).
	\]
These restrictions are indecomposable simplices of the nerves
\[
	N\mathcal C_P(z_{j-1},z_j),
	\]
possibly degenerate; their nondegenerate reductions are indecomposable by construction. Thus each such nondegenerate reduction has smaller dimension and has already been attached. This implies that the characteristic map \( \chi_{\partial_0\sigma} \) is already defined. It follows that every face of \(\sigma\) has already been constructed by the time the cell corresponding to \(\sigma\) is attached.

Let \(G^\omega(P) = \injlim G^m(P)\) be the flow obtained by the above cellular construction. We claim that
\[
	G^\omega(P)=\Ch(P).
	\]
By Lemma~\ref{lem:normal-form}, the map \(G^\omega(P) \to \Ch(P)\) induces injective maps on each path space. Indeed, all cells used in the construction are indecomposable nondegenerate simplices of the nerves \(N\mathcal C_P(x,y)\), and the composition is induced by concatenation of chains. This follows from the uniqueness of the normal form.

Conversely, let \(\tau\) be any nondegenerate simplex of \(N\mathcal C_P(x,y)\). If \(\tau\) is indecomposable, then it is one of the attached cells. If it is not indecomposable, write
\[
	V_P(\tau)=\{x=z_0<z_1<\cdots<z_r=y\}.
	\]
Then every chain in \(\tau\) contains each \(z_j\), and \(\tau\) is the composite of its indecomposable restrictions to the intervals \([z_{j-1},z_j]\). The nondegenerate reductions of these restrictions are indecomposable. Hence \(\tau\) is obtained by composition from the attached cells. Therefore the map of flows \(G^\omega(P)\to\Ch(P)\) induces a bijection on the set of states and a continuous bijection
\[f_{x,y}:\PP_{x,y}G^\omega(P)\longrightarrow \PP_{x,y}\Ch(P)\]
for each pair of states \((x,y)\).

Each \(\PP_{x,y}\Ch(P)\) is the geometric realization of a simplicial nerve, namely \(N\mathcal C_P(x,y)\). Each \(m\)-simplex \(\sigma\) is indecomposable if and only if its nondegenerate reduction \(\tau_\sigma\) is indecomposable. When \(\sigma\) is indecomposable, including the degenerate case, we obtain the commutative diagram of sets
\[
	\begin{tikzcd}
		{|\Delta^m|}\arrow[d,"|s_\sigma|"'] \arrow[r,"|\sigma|"]& \PP_{x,y}\Ch(P)\arrow[r,"f_{x,y}^{-1}"]& \PP_{x,y}G^\omega(P)\\
		{|\Delta^{n_\sigma}|} \arrow[ru,"|\tau_\sigma|"'] \arrow[rru,bend right,"\PP_{0,1}\chi_{\tau_\sigma}" pos=0.7]&&
	\end{tikzcd}
	\]
where \((s_\sigma,\tau_\sigma)\) is as above. We have proved that for any indecomposable simplex \(\sigma\), including the degenerate case, the composite map
\[
	p_\sigma:|\Delta^m|\xlongrightarrow{|\sigma|} \PP_{x,y}\Ch(P) \xlongrightarrow{f_{x,y}^{-1}} \PP_{x,y}G^\omega(P)
	\]
is equal to \(\PP_{0,1}\chi_\sigma\), and therefore is continuous. When \(\sigma\) is decomposable, write
\[
	\sigma=\sigma_1*\dots *\sigma_p,\quad p>1, \quad \sigma_1,\cdots,\sigma_p\hbox{ indecomposable}.
	\]
Let \(x=z_0<\cdots<z_p=y\) be the common vertices of the chains in \(\sigma\). The formula \(p_\sigma=p_{\sigma_1}*\cdots*p_{\sigma_p}\) gives the commutative diagram
\[
\begin{tikzcd}[column sep=large,row sep=large]
{|\Delta^m|} \arrow[r,"\delta"] \arrow[d,"p_\sigma"'] & \displaystyle\prod_{j=1}^p|\Delta^m| \arrow[d,"\prod_j p_{\sigma_j}"] \\
\PP_{x,y}G^\omega(P) & \displaystyle\prod_{j=1}^p\PP_{z_{j-1},z_j}G^\omega(P) \arrow[l,"*"'] .
\end{tikzcd}
\]
Here \(\delta\) is the diagonal and \(*\) is composition. All three maps in this factorization are continuous, so \(p_\sigma\) is continuous.

The topology of \(\PP_{x,y}\Ch(P)\) is the final topology with respect to each \(|\sigma|:|\Delta^m|\to \PP_{x,y}\Ch(P)\) for \(\sigma\) running over all nondegenerate simplices of \(N\mathcal C_P(x,y)\). Since each composite
\[
	|\Delta^m|\xlongrightarrow{|\sigma|} \PP_{x,y}\Ch(P) \xlongrightarrow{f_{x,y}^{-1}} \PP_{x,y}G^\omega(P)
	\]
is continuous, in both cases, we deduce that each \(f_{x,y}\) is a homeomorphism. We have proved that \(G^\omega(P)=\Ch(P)\). Hence \(\Ch(P)\) is a cell complex, and in particular it is q-cofibrant.

For \(x<y\), the poset \(\mathcal C_P(x,y)\) has the initial object \((x<y)\). Consequently
\[
	|N\mathcal C_P(x,y)|
	\]
is contractible. The map
\[
	\epsilon_P:\Ch(P)\longrightarrow P
	\]
is the identity on states and sends each nonempty contractible path space to the singleton path space of the poset flow \(P\). If \(x\not<y\), both path spaces are empty. Thus every map on path spaces induced by \(\epsilon_P\) is a trivial q-fibration of spaces. Since the map on states is a bijection, \(\epsilon_P\) is a trivial q-fibration of flows.

It remains to prove the last assertion. Let \(P\subset Q\) be an inclusion of posets. The cellular construction just described for \(P\) is a subcomplex of the corresponding cellular construction for \(Q\). Indeed, a simplex of \(N\mathcal C_P(x,y)\) has the same set of vertices whether it is regarded in \(P\) or in \(Q\). Therefore \(\Ch(Q)\) is obtained from \(\Ch(P)\) by first adding the new states of \(Q\setminus P\), and then attaching the remaining globular cells of the cellular decomposition of \(\Ch(Q)\). Hence
\[
	\Ch(P)\longrightarrow \Ch(Q)
	\]
is a q-cofibration of flows.
\end{proof}

\begin{samepage}
\begin{proposition}\label{prop:restriction-cofibration}
Let \(P\subset Q\) be an inclusion of posets.  Then the canonical map which is bijective on states
\[
	\Ch(P)\longrightarrow \Ch(Q)|_P
	\]
is a q-cofibration between q-cofibrant flows.  More precisely, it is a relative cell complex.
\end{proposition}
\end{samepage}

\begin{proof}
The sets of states of \(\Ch(P)\) and \(\Ch(Q)|_P\) are both equal to \(P\). For \(x,y\in P\), the map on path spaces is the canonical map
\[
	|N\mathcal C_P(x,y)|
	\longrightarrow
	|N\mathcal C_Q(x,y)|,
	\]
where the source is empty if \(x\not<y\) in \(P\). We shall construct \(\Ch(Q)|_P\) from \(\Ch(P)\) by successively attaching cells. The proof is similar to the proof of Proposition~\ref{prop:chain-replacement}.

Let \(x,y\in P\), and let
\[
	\sigma=(\gamma_0<\gamma_1<\cdots<\gamma_m)
	\]
be a nondegenerate \(m\)-simplex of \(N\mathcal C_Q(x,y)\). Put
\[
	V_P(\sigma)
	=
	P\cap \bigcap_{i=0}^{m}V(\gamma_i) = P\cap V(\gamma_0).
	\]
Thus \(V_P(\sigma)\) is the set of vertices of \(P\) through which every chain \(\gamma_i\) passes. It contains \(x\) and \(y\). We say that \(\sigma\) is \textit{\(P\)-indecomposable} if
\[
	V_P(\sigma)=\{x,y\}.
	\]
Equivalently, no intermediate state \(z\in P\), \(x<z<y\), occurs in every chain \(\gamma_i\). Thus if \(\sigma\) is indecomposable, then it is \(P\)-indecomposable. The converse implication is false in general unless \(P=Q\). We call \(\sigma\) \emph{\(P\)-decomposable} if it is not \(P\)-indecomposable.

Starting from \(G^{-1}(P,Q)=\Ch(P)\), attach, in increasing dimension, one globular \(m\)-cell from \(x\) to \(y\) for every \(P\)-indecomposable nondegenerate \(m\)-simplex of \(N\mathcal C_Q(x,y)\) which is not already a simplex of \(N\mathcal C_P(x,y)\):
\[
	\begin{tikzcd}
		\displaystyle\coprod_{\sigma} \Glob(|\partial\Delta^m|)\arrow[r]\arrow[d] & G^{m-1}(P,Q) \arrow[d]\arrow[r]&\Ch(Q) \arrow[d,equal]\\
		\displaystyle\coprod_{\sigma} \Glob(|\Delta^m|) \arrow[r]\arrow[rr,bend right=2em,"\coprod_\sigma \Psi_\sigma"' pos=0.45] & G^{m}(P,Q)\cocartesian \arrow[r,dashed]&\Ch(Q),
	\end{tikzcd}
	\]
where, for each \(P\)-indecomposable nondegenerate \(m\)-simplex \(\sigma\) of \(N\mathcal C_Q(x,y)\backslash N\mathcal C_P(x,y)\), \(|\sigma|:|\Delta^m|\to |N\mathcal C_Q(x,y)|\) induces the map \(\Psi_\sigma:\Glob(|\Delta^m|) \to \Ch(Q)\) such that \(|\sigma|=\PP_{0,1}\Psi_\sigma\). The universal property of the pushout gives the dashed arrow.

We check that the attaching map of each such cell is already defined. Let \(\sigma\) be a \(P\)-indecomposable nondegenerate \(m\)-simplex of \(N\mathcal C_Q(x,y)\), not belonging to \(N\mathcal C_P(x,y)\), and let \(\partial_i\sigma\) be a face of codimension one. If \(\partial_i\sigma\) belongs to \(N\mathcal C_P(x,y)\), then it has already been constructed in \(\Ch(P)\). If \(\partial_i\sigma\) is \(P\)-indecomposable and does not belong to \(N\mathcal C_P(x,y)\), then it has dimension \(m-1\), and so it has already been attached. It remains to consider the case where \(\partial_i\sigma\) is \(P\)-decomposable (which implies \(i=0\)). Write
\[
	V_P(\partial_0\sigma)=\{x=z_0<z_1<\cdots<z_r=y\}.
	\]
Every chain in \(\partial_0\sigma\) contains each \(z_j\). Hence \(\partial_0\sigma\) is the composite of its restrictions to the intervals
\[
	[z_{j-1},z_j]
	\qquad (1\leq j\leq r).
	\]
Each restriction is a possibly degenerate simplex of
\[
	N\mathcal C_Q(z_{j-1},z_j).
	\]
Its nondegenerate reduction is \(P\)-indecomposable. Indeed, a common intermediate vertex of \(P\) in such a restriction would be an element of \(V_P(\partial_0\sigma)\) strictly between two consecutive elements \(z_{j-1}\) and \(z_j\), which is impossible. Thus each nondegenerate reduction either lies in \(N\mathcal C_P(z_{j-1},z_j)\), and has already been constructed in \(\Ch(P)\), or is one of the \(P\)-indecomposable cells already attached in lower dimension. Therefore every face of \(\sigma\) has already been constructed, and the attaching map is well defined.

Let \(G^\omega(P,Q) = \injlim G^m(P,Q)\) be the flow obtained after all these cells have been attached. We claim that
\[
	G^\omega(P,Q)=\Ch(Q)|_P.
	\]
The \(P\)-relative analogue of Lemma~\ref{lem:normal-form}, obtained by replacing indecomposable simplices with \(P\)-indecomposable simplices, shows that the canonical morphism
\[
	G^\omega(P,Q)\longrightarrow \Ch(Q)|_P
	\]
induces injective maps on all path spaces. Indeed, the attached cells correspond precisely to \(P\)-indecomposable nondegenerate simplices of \(N\mathcal C_Q(x,y)\), with \(x,y\in P\), and composition corresponds to concatenation of chains in \(Q\). This follows from the uniqueness of the normal form.

Conversely, let \(\tau\) be a nondegenerate simplex of \(N\mathcal C_Q(x,y)\) with \(x,y\in P\). If \(\tau\) belongs to \(N\mathcal C_P(x,y)\), then it has already been constructed in \(\Ch(P)\). If \(\tau\) is \(P\)-indecomposable and does not belong to \(N\mathcal C_P(x,y)\), then it has been attached. Otherwise, \(\tau\) is \(P\)-decomposable. Writing
\[
	V_P(\tau)=\{x=z_0<z_1<\cdots<z_r=y\},
	\]
the simplex \(\tau\) is the composite of its restrictions to the intervals \([z_{j-1},z_j]\). The nondegenerate reductions of these restrictions are \(P\)-indecomposable by the same argument as above; hence they either lie in \(\Ch(P)\) or are among the attached cells. Thus \(\tau\) belongs to the path space of \(G^\omega(P,Q)\). Therefore every simplex of every \(N\mathcal C_Q(x,y)\), with \(x,y\in P\), belongs to \(G^\omega(P,Q)\), and so, using the final topology as in Proposition~\ref{prop:chain-replacement}, we deduce the equality
\[
	G^\omega(P,Q)=\Ch(Q)|_P.
	\]
Thus
\[
	\Ch(P)\longrightarrow \Ch(Q)|_P
	\]
is a relative cell complex, hence a q-cofibration of flows. Since \(\Ch(P)\) is a cell complex by Proposition~\ref{prop:chain-replacement} and \(\Ch(Q)|_P\) has just been obtained from \(\Ch(P)\) by attaching cells, both flows are q-cofibrant.
\end{proof}

\begin{samepage}
\begin{thm}\label{thm:order-reflecting-effect}
Let \(P\subset Q\) be an inclusion of posets.  Then the following conditions are equivalent:
\begin{enumerate}[label=\arabic*.]
\item The inclusion \(P\subset Q\) is order-reflecting.
\item The canonical map which is bijective on states
\[
		\Ch(P)\longrightarrow \Ch(Q)|_P
		\]
is a trivial q-cofibration between q-cofibrant flows.
\end{enumerate}
\end{thm}
\end{samepage}

\begin{proof}
Assume that the inclusion \(P\subset Q\) is order-reflecting. By Proposition~\ref{prop:restriction-cofibration}, the map \(\Ch(P)\to \Ch(Q)|_P\) is a q-cofibration. It remains only to verify that this q-cofibration is a weak equivalence. If \(x\not<y\) in \(P\), then \(x\not<y\) in \(Q\), because \(P\subset Q\) is order-reflecting.  Both path spaces are therefore empty.  If \(x<y\) in \(P\), then also \(x<y\) in \(Q\). The initial objects in the two posets of chains give the commutative triangle
\[
\begin{tikzcd}[column sep=large,row sep=large]
{|N\mathcal C_P(x,y)|} \arrow[r] \arrow[dr,"\simeq"'] & {|N\mathcal C_Q(x,y)|} \arrow[d,"\simeq"] \\
& \{*\}.
\end{tikzcd}
\]
The maps to a point are homotopy equivalences. Hence the horizontal map is a weak homotopy equivalence, and \(\Ch(P)\to\Ch(Q)|_P\) is a weak equivalence of flows.

Assume now that the inclusion \(P\subset Q\) is not order-reflecting. Then there exist \(x,y\in P\) such that \(x\not<y\) in \(P\) and \(x<y\) in \(Q\). Then \(\PP_{x,y}\Ch(P)=\varnothing\) and \(\PP_{x,y}\Ch(Q)|_P\) is nonempty. Thus the map \(\Ch(P)\to \Ch(Q)|_P\) cannot be a weak equivalence of the q-model structure of \(\Flow\).
\end{proof}

\begin{figure}[htbp]
\centering
\begin{tikzpicture}[x=1.05cm,y=1cm,>=Stealth]
\node at (0,1.5) {$P$};
\node at (4,1.5) {$Q$};
\fill (-.65,.5) circle (2pt) node[below=4pt] {$x$};
\fill (.65,.5) circle (2pt) node[below=4pt] {$y$};
\draw[->] (1.3,.5)--(2.3,.5) node[midway,above] {$i$};
\foreach \xx/\lab in {2.8/x,4/z,5.2/y} {
  \fill (\xx,.5) circle (2pt) node[below=4pt] {$\lab$};
}
\draw[->,shorten >=4pt,shorten <=4pt] (2.8,.5)--(4,.5);
\draw[->,shorten >=4pt,shorten <=4pt] (4,.5)--(5.2,.5);
\node at (0,-.7) {$\PP_{x,y}\Ch(P)=\varnothing$};
\node at (4,-.7) {$\PP_{x,y}\Ch(Q)\cong[0,1]$};
\end{tikzpicture}
\caption{An inclusion which is not order-reflecting. In \(P\), the two elements are incomparable. In \(Q\), the two chains \((x<y)\) and \((x<z<y)\) form the poset \(\mathcal C_Q(x,y)\). Its nerve has one nondegenerate 1-simplex.}
\label{fig:order-reflecting}
\end{figure}

\section{Full subflows and pushouts}
\label{sec:subflow-lemma}

Write \(\Top_\Delta\) for the category of \(\Delta\)-generated spaces, \(\mathrm H\Top_\Delta\) for its full subcategory of \(\Delta\)-Hausdorff spaces, and \(\TOP_\Delta(-,-)\) for the internal hom of \(\Top_\Delta\). For either category \(\mathcal T\), let \(\Flow(\mathcal T)\) denote the category of small semicategories enriched over \(\mathcal T\).

\begin{samepage}
\begin{proposition}\label{prop:flow-reflection}
The reflector \(w_\Delta:\Top_\Delta\to\mathrm H\Top_\Delta\) preserves finite products. Applying it to each space of execution paths, without changing the states, defines a left adjoint
\[
\widehat w_\Delta:\Flow(\Top_\Delta)\longrightarrow\Flow(\mathrm H\Top_\Delta)
\]
to the inclusion. For every \(F\in\Flow(\Top_\Delta)\) and every subset \(A\subset F^0\),
\[
\widehat w_\Delta(F|_A)\cong(\widehat w_\Delta F)|_A.
\]
\end{proposition}
\end{samepage}

\begin{proof}
Let \(\eta_K:K\to w_\Delta K\) be the reflection unit \cite[Proposition~B.7]{leftproperflow}. Cartesian closedness, reflection, and \cite[Proposition~B.6]{leftproperflow} give, for \(K,L\in\Top_\Delta\) and \(Z\in\mathrm H\Top_\Delta\),
\[
\begin{aligned}
\Top_\Delta(K\times L,Z)
&\cong\Top_\Delta(w_\Delta L,\TOP_\Delta(K,Z))\\
&\cong\Top_\Delta(w_\Delta K,\TOP_\Delta(w_\Delta L,Z))\\
&\cong\Top_\Delta(w_\Delta K\times w_\Delta L,Z).
\end{aligned}
\]
Products are formed in \(\Top_\Delta\); those of \(\Delta\)-Hausdorff spaces remain \(\Delta\)-Hausdorff by \cite[Proposition~B.4]{leftproperflow} and \(\Delta\)-kelleyfication. Thus the canonical comparison is an isomorphism:
\[
\begin{gathered}
\kappa_{K,L}:w_\Delta(K\times L)\xrightarrow{\cong}w_\Delta K\times w_\Delta L,\\
\kappa_{K,L}\eta_{K\times L}=\eta_K\times\eta_L,
\qquad w_\Delta\{*\}\cong\{*\}.
\end{gathered}
\]
Write \(F_{x,y}=\PP_{x,y}F\) and \(\mu_{x,y,z}:F_{x,y}\times F_{y,z}\to F_{x,z}\) for composition. On \(\widehat w_\Delta F\), put
\[
\begin{gathered}
\widehat\mu_{x,y,z}:w_\Delta F_{x,y}\times w_\Delta F_{y,z}\longrightarrow w_\Delta F_{x,z},\\
\widehat\mu_{x,y,z}=w_\Delta(\mu_{x,y,z})\circ\kappa^{-1}_{F_{x,y},F_{y,z}},\\
\widehat\mu_{x,y,z}(\eta_{F_{x,y}}\times\eta_{F_{y,z}})=\eta_{F_{x,z}}\mu_{x,y,z}.
\end{gathered}
\]
Precomposition with finite products of units is injective for \(\Delta\)-Hausdorff targets. Hence the reflected composition is associative and factorizations of path maps preserve composition, giving, for \(G\in\Flow(\mathrm H\Top_\Delta)\),
\[
\Flow(\mathrm H\Top_\Delta)(\widehat w_\Delta F,G)
\cong\Flow(\Top_\Delta)(F,G).
\]
Restriction commutes by definition.
\end{proof}

Let \(\mathsf{Graph}(\mathcal T)\) consist of directed graphs with vertices \(G^0\) and edge spaces \(G_{x,y}\in\mathcal T\). A morphism \(f:G\to H\) is a function \(f^0:G^0\to H^0\) together with continuous maps \(f_{x,y}:G_{x,y}\to H_{f^0(x),f^0(y)}\). The edge space \(G^1=\coprod_{x,y}G_{x,y}\) has source and target in the discrete space \(G^0\). Write \(\mathcal U:\Flow(\mathcal T)\to\mathsf{Graph}(\mathcal T)\) for the functor forgetting composition.

\begin{samepage}
\begin{proposition}\label{prop:free-flow}
For either choice of \(\mathcal T\), the functor \(\mathcal U\) has a left adjoint \(\mathcal F\), the free flow functor, with \((\mathcal F G)^0=G^0\) and
\[
\PP_{x,y}\mathcal F G=
\coprod_{n\geq 1}\ \coprod_{x_1,\ldots,x_{n-1}\in G^0}
\prod_{i=1}^{n}G_{x_{i-1},x_i},
\qquad x_0=x,\quad x_n=y.
\]
Composition is induced by concatenation.
\end{proposition}
\end{samepage}

\begin{proof}
By cartesian closedness, products in \(\mathcal T\) distribute over coproducts, so concatenation is continuous and associative. The summands with \(n=1\) define \(\iota_G:G\to\mathcal U\mathcal F G\), with identity map on vertices.

A graph morphism \(f:G\to\mathcal U X\) extends with \(\overline f^0=f^0\) and
\[
\left.\PP_{x,y}\overline f\right|_{\prod_{i=1}^{n}G_{x_{i-1},x_i}}
=*^{(n)}\circ\prod_{i=1}^{n}f_{x_{i-1},x_i},
\]
where \(*^{(n)}\) is iterated composition in \(X\), with \(*^{(1)}=\operatorname{Id}\). These continuous maps preserve composition by associativity. Since words are composites of letters, the extension is unique:
\[
\begin{tikzcd}[column sep=large,row sep=large]
G \arrow[r,"\iota_G"] \arrow[dr,"f"'] & \mathcal U\mathcal F G \arrow[d,dashed,"\mathcal U\overline f"] \\
& \mathcal U X.
\end{tikzcd}
\]
This proves \(\mathcal F\dashv\mathcal U\).
\end{proof}

\begin{samepage}
\begin{proposition}\label{prop:restriction-pushout}
Every pushout of flows on the left induces a pushout on the right, where \(U'=U|_{\phi(V^0)}\) and \(Y'=Y|_{\psi(X^0)}\):
\[
	\begin{tikzcd}[column sep=large]
		V \arrow[r,"\alpha"] \arrow[d,"\phi"'] & X \arrow[d,"\psi"] \\
		U \arrow[r,"\theta"'] & Y \cocartesian
	\end{tikzcd}
\quad\Longrightarrow\quad
	\begin{tikzcd}[column sep=large]
		V \arrow[r,"\alpha"] \arrow[d,"\phi"'] & X \arrow[d] \\
		U' \arrow[r] & Y' \cocartesian .
	\end{tikzcd}
	\]
\end{proposition}
\end{samepage}

\begin{proof}
First suppose that \(\Top=\Top_\Delta\). Put
\[
S=\phi(V^0),\qquad T=U^0\amalg_{V^0}X^0=Y^0,\qquad B=\psi(X^0).
\]
The relations \(\phi(v)\sim\alpha(v)\), \(v\in V^0\), leave \(U^0\setminus S\) untouched. Hence
\[
T\cong B\amalg(U^0\setminus S),\qquad
B\cong X^0\amalg_{V^0}S,\qquad
(\theta^0)^{-1}(B)=S.
\]
Let \(G\) have vertices \(T\), edges \(\PP X\amalg\PP U\), and endpoint maps induced by \(\psi^0,\theta^0\). By Proposition~\ref{prop:free-flow}, distributing the edge coproducts expresses \(W_{p,q}=\PP_{p,q}\mathcal F G\) as a coproduct of products of original path spaces, with sources and targets fixed. Let \(\sim\) be the congruence generated by composing paths composable in \(X\) or \(U\), and by \(\alpha(\nu)\sim\phi(\nu)\), \(\nu\in\PP V\).

Products and quotients are formed in \(\Top_\Delta\). Its inclusion into ordinary topological spaces is left adjoint to \(\Delta\)-kelleyfication \(k_\Delta\) \cite[Appendix~B, p.~603]{leftproperflow}, so coequalizers have the ordinary final topology. A quotient \(q:K\to L\) is the coequalizer of \({k_\Delta(K\times_L^{\mathrm{TOP}}K)\rightrightarrows K}\): kelleyfication leaves the underlying relation unchanged, and finality gives continuous factorization. The left adjoint \(({-})\times K\) preserves coequalizers, so products of quotient maps are quotient maps. Thus concatenation descends continuously:
\[
\begin{tikzcd}[column sep=large,row sep=large]
W_{p,q}\times W_{q,r} \arrow[r,"\text{concatenation}"] \arrow[d,"\pi_{p,q}\times\pi_{q,r}"'] & W_{p,r} \arrow[d,"\pi_{p,r}"] \\
(W_{p,q}/{\sim})\times(W_{q,r}/{\sim}) \arrow[r,dashed,"*"'] & W_{p,r}/{\sim},
\end{tikzcd}
\]
Here \(\pi_{p,q}\) is the quotient map. The adjunction \(\mathcal F\dashv\mathcal U\) identifies the quotient flow with \(Y\), so
\[
\PP_{p,q}Y\cong W_{p,q}/{\sim}.
\]
Let \(a:X\to Z\) and \(b:U'\to Z\) satisfy \(a\alpha=b\phi\), inducing \(\ell^0:B\to Z^0\). Cut words with endpoints in \(B\) at every intermediate state in \(B\). Each piece is one path of \(X\) or consists of paths of \(U\) with internal states in \(T\setminus B\). By the state identifications above, the latter paths are composable in \(U\) with endpoints in \(S\). Define \(E\) by applying \(a\) to paths of \(X\), \(b\) to the composites in \(U'\), and composing in \(Z\). Since \(a^0,b^0\) factor through \(\ell^0\), adjacent images compose in \(Z\) even when distinct original endpoints have the same image in \(B\).

Composition in \(X\) is respected by \(a\). Composing in \(U\) preserves a piece's composite when the junction is outside \(S\), or joins two pieces with matching original endpoints when it is in \(S\); \(b\) respects the latter composition. For \(\nu\in\PP V\), both images are single pieces because their endpoints map into \(B\), so
\[
\begin{aligned}
E(w_1*\phi(\nu)*w_2)
&=E(w_1)*b(\phi(\nu))*E(w_2)\\
&=E(w_1)*a(\alpha(\nu))*E(w_2)\\
&=E(w_1*\alpha(\nu)*w_2).
\end{aligned}
\]
Factors with empty \(w_i\) are omitted. The composition equalities hold on words before descending \(E\) to the quotient. By symmetry, reversing composition relations along actual factorizations in \(X\) or \(U\) exposes factors \(\alpha(\nu)\) or \(\phi(\nu)\), respectively, without changing \(E\). Thus \(E\) respects \(\sim\).

On each product summand, the original sources, targets and cuts are fixed, so \(E\) is continuous. Since joins lie in \(B\), \(E(w*w')=E(w)*E(w')\). Every word is a composite of its pieces, so the induced morphism \(\ell:Y'\to Z\) is the unique extension of \(a,b\).

For \(\Top=\mathrm H\Top_\Delta\), form the pushout \(\widetilde Y\) in \(\Flow(\Top_\Delta)\), identify \(\widetilde Y^0=Y^0=U^0\amalg_{V^0}X^0\), and put \(B=\psi(X^0)\). Proposition~\ref{prop:flow-reflection} and the first case give
\[
Y'\cong\widehat w_\Delta(\widetilde Y|_B)\cong X\amalg_V U',
\]
where the last pushout is formed in \(\Flow(\mathrm H\Top_\Delta)\).
\end{proof}

\begin{samepage}
\begin{proposition}[Full subflow pushout lemma]\label{prop:full-subflow-pushout}
Let \(U\) be a flow and \(S\subset U^0\). For a morphism \(g:U|_S\to X\), form the pushout
\[
	\begin{tikzcd}
		U|_{S} \arrow[r,"g"] \arrow[d,hook] & X \arrow[d,"\iota"]\\
		U \arrow[r] & Y \cocartesian .
	\end{tikzcd}
	\]
The canonical maps are isomorphisms:
\[
\begin{gathered}
X\xrightarrow{\cong}Y|_{\iota(X^0)},\\
\PP_{\alpha,\beta}X\xrightarrow{\cong}\PP_{\iota(\alpha),\iota(\beta)}Y
\qquad(\alpha,\beta\in X^0).
\end{gathered}
\]
\end{proposition}
\end{samepage}

\begin{proof}
Proposition~\ref{prop:restriction-pushout} gives \(Y|_{\iota(X^0)}\cong X\amalg_{U|_S}U|_S\cong X\).
\end{proof}

\begin{samepage}
\begin{cor}
Let \(P\subset Q\) be an order-reflecting inclusion of posets. Consider a pushout of flows
\[
	\begin{tikzcd}
		P \arrow[r] \arrow[d] & X \arrow[d,"\iota"]\\
		Q \arrow[r] & Y \cocartesian .
	\end{tikzcd}
	\]
For all states \(\alpha,\beta\in X^{0}\), the induced map
\[
	\PP_{\alpha,\beta}X
	\longrightarrow
	\PP_{\iota(\alpha),\iota(\beta)}Y
	\]
is a homeomorphism.
\end{cor}
\end{samepage}

\begin{proof}
Order reflection gives \(P=Q|_P\); apply Proposition~\ref{prop:full-subflow-pushout}.
\end{proof}

\section{Pushout along a chain replacement}
\label{sec:pushout-along-chain-replacement}

\begin{samepage}
\begin{definition} [{\cite[p. 686]{hocolimfacile}}] An injective continuous map \(i:A\rightarrow X\) of topological spaces is \textit{relative-\(T_1\)} if for any open subset \(U\) of \(A\) and any point \(z\in X\backslash i(U)\), there is an open set \(W\) of \(X\) with \(i(U)\subset W\) and \(z\notin W\).
\end{definition}
\end{samepage}

Every closed-\(T_1\) inclusion is relative-\(T_1\) by \cite[Proposition~2.3]{leftproperflow}: we shall use the following result.

\begin{samepage}
\begin{thm} \label{thm:path-almost-accessible}
Let \(X:\lambda \to \Flow\) be a transfinite tower of flows such that for all \(\mu<\lambda\), the map \(\PP X_\mu \to \PP X_{\mu+1}\) is a relative-\(T_1\) inclusion. Then the canonical map
\[\injlim (\PP\cdot X) \longrightarrow \PP \injlim X\]
is a homeomorphism. Moreover the topology of \(\PP \injlim X\) is the final topology.
\end{thm}
\end{samepage}

\begin{proof}
It is \cite[Theorem~5.5]{leftproperflow}.
\end{proof}

Theorem~\ref{thm:path-almost-accessible} applies in particular when each map \(\PP X_\mu \to \PP X_{\mu+1}\) is an h-cofibration since every h-cofibration of \(\Top\) is relative-\(T_1\) by \cite[Proposition~2.6]{leftproperflow}.

\begin{samepage}
\begin{proposition} \label{prop:q-2-h}
Let \(f:X\to Y\) be a trivial q-cofibration of flows. Then for every \(u,v\in X^0\), the induced map
\[
	\PP_{u,v} X \xlongrightarrow{f}\PP_{f(u),f(v)}Y
	\]
is a trivial h-cofibration of spaces.
\end{proposition}
\end{samepage}

When \(X\) and \(Y\) are q-cofibrant, the result follows from \cite[Theorem~5.7]{leftproperflow}. The following proof does not require this assumption.

\begin{proof}
Every trivial q-cofibration is a retract of a relative \(J^{\mathrm{gl}}\)-cell complex. Since trivial h-cofibrations are closed under retracts, it suffices to consider a transfinite composition of pushouts of maps in \(J^{\mathrm{gl}}\):
\[
X=X_0 \longrightarrow X_1 \longrightarrow \dots \longrightarrow \injlim X_\lambda = Y.
\]
By \cite[Theorem~5.4(2)]{leftproperflow}, each pushout of a generating trivial q-cofibration of flows induces a trivial h-cofibration on the path spaces (i.e. a trivial cofibration of the h-model structure of \(\Top\)). Thus we obtain a transfinite tower of trivial h-cofibrations of spaces
\[
\PP_{u,v}X=\PP_{u,v}X_0 \longrightarrow \PP_{u,v} X_1 \longrightarrow \dots \longrightarrow \injlim \PP_{u,v} X_\lambda\, .
\]
A transfinite composite of trivial h-cofibrations is again a trivial h-cofibration since they are trivial cofibrations of the h-model structure of \(\Top\). Since every h-cofibration of \(\Top\) is relative-\(T_1\)  by \cite[Proposition~2.6]{leftproperflow}, the homeomorphism of Theorem~\ref{thm:path-almost-accessible} induces the map
\[
\PP_{u,v}X\longrightarrow \PP_{f(u),f(v)}Y
\]
which is the composite of a trivial h-cofibration and a homeomorphism.
\end{proof}

\begin{samepage}
\begin{thm} \label{thm:canonical}
Let \(i:P\subset Q\) be an order-reflecting inclusion of posets, and let
\[
\Ch(i):\Ch(P)\longrightarrow \Ch(Q)
\]
be the chain replacement of the inclusion \(P\subset Q\).  Let
\[
\begin{tikzcd}
\Ch(P) \ar[r] \ar[d,"\Ch(i)"',rightarrowtail] & X \ar[d,"f",rightarrowtail] \\
\Ch(Q) \ar[r] & Y \cocartesian
\end{tikzcd}
\]
be a pushout of flows.  Then for all states \(u,v\in X^0\), the induced map
\[
\PP_{u,v}X\longrightarrow \PP_{f(u),f(v)}Y
\]
is a trivial h-cofibration of topological spaces.
\end{thm}
\end{samepage}

\begin{proof}
Factor \(\Ch(i)\) through \(\Ch(Q)|_P\), and form the two pushout squares
\[
\begin{tikzcd}[column sep=large,row sep=large]
\Ch(P) \arrow[r] \arrow[d,rightarrowtail,"\simeq"'] & X \arrow[d,rightarrowtail,"\simeq"] \\
\Ch(Q)|_P \arrow[r] \arrow[d,hook] & X_1 \cocartesian \arrow[d] \\
\Ch(Q) \arrow[r] & Y \cocartesian .
\end{tikzcd}
\]
Thus \(X_1=X\amalg_{\Ch(P)}\Ch(Q)|_P\) and \(Y\cong X_1\amalg_{\Ch(Q)|_P}\Ch(Q)\). By Theorem~\ref{thm:order-reflecting-effect}, the induced map
\[
X\longrightarrow X_1
\]
is a trivial q-cofibration of flows. By Proposition~\ref{prop:q-2-h}, for all \(u,v\in X^0\), the map
\[
\PP_{u,v}X\longrightarrow \PP_{u,v}X_1
\]
is a trivial h-cofibration. Apply Proposition \ref{prop:full-subflow-pushout} with
\[
U=\Ch(Q),
\qquad
S=P,
\qquad
U|_S=\Ch(Q)|_P,
\qquad
X=X_1.
\]
It gives the isomorphism of flows
\[
X_1\cong Y|_{f(X^0)}.
\]
For \(u,v\in X^0=X_1^0\), the induced maps form the commutative triangle
\[
\begin{tikzcd}[column sep=large,row sep=large]
\PP_{u,v}X \arrow[r,rightarrowtail,"\simeq"] \arrow[dr] & \PP_{u,v}X_1 \arrow[d,"\cong"] \\
& \PP_{f(u),f(v)}Y.
\end{tikzcd}
\]
The horizontal map is a trivial h-cofibration and the vertical map is a homeomorphism. Their composite is therefore a trivial h-cofibration.
\end{proof}

\section{The H-model structure of flows}
\label{sec:H-model-structure}

\begin{samepage}
\begin{lem}[well-known]\label{lem:components}
Every object \(U\) of \(\Top\) is canonically homeomorphic, in \(\Top\), to the coproduct of its path-connected components:
\[
	\coprod_{c\in\pi_0(U)} U_c \xrightarrow{\ \cong\ } U.
	\]
\end{lem}
\end{samepage}

\begin{proof}
See \cite[Proposition~2.8]{mdtop}.
\end{proof}

The following consequence will be used repeatedly.  For every \(A\in\Top\), Lemma~\ref{lem:components} gives a natural homeomorphism
\begin{equation}\label{eq:mapsfromcomponentcoproduct}
	\TOP(U,A)
	\cong
	\prod_{c\in\pi_0(U)} \TOP(U_c,A).
\end{equation}
This follows from the universal property of the coproduct \(U\cong\coprod_{c}U_c\).

We denote by \(\FLOW(X,Y)\) the usual enriched hom space of morphisms of flows from \(X\) to \(Y\). Thus the category of flows is regarded with its standard \(\Top\)-enrichment. By \cite[Theorem~5.10~(1)]{model3}, the comparison map
\[
\FLOW(\injlim X_i,Y) \xlongrightarrow{\cong} \projlim \FLOW(X_i,Y)
\]
is a homeomorphism. In the \(\Delta\)-generated setting, and unlike the \(k\)-space setting, the functor \(\widetilde{\omega}\) of \cite[Theorem~5.4]{model3} does commute with all limits, and not only with finite limits, because a limit of discrete spaces remains discrete in \(\Top\). Thus the same calculation as in the proof of  \cite[Theorem~5.10~(2)]{model3} proves that the comparison map
\[
\FLOW(X,\projlim Y_i) \xlongrightarrow{\cong} \projlim \FLOW(X,Y_i)
\]
is a homeomorphism for any limit.

For \(U\in\Top\) and a flow \(Y\), recall the flow \(\{U,Y\}_S\):
\[
\{U,Y\}_S^0=Y^0,
\qquad
\PP_{\alpha,\beta}\{U,Y\}_S=\TOP(U,\PP_{\alpha,\beta}Y),
\]
with pointwise composition.  Theorem~7.8 of \cite{model3} says that the functor \(\{U,-\}_S:\Flow\to\Flow\) has a left adjoint, denoted by
\[
U\boxtimes - :\Flow\longrightarrow\Flow.
\]
In particular, there are natural bijections
\[
\Flow(U\boxtimes X,Y)\cong \Flow(X,\{U,Y\}_S).
\]
When \(C\) is non-empty and path-connected, there are natural bijections of sets
\[
\Top(C,\FLOW(X,Y))
\cong
\Flow(C\boxtimes X,Y)
\cong
\Flow(X,\{C,Y\}_S)
\]
by \cite[Theorem~7.9]{model3}. We shall use the standard strengthened form for a non-empty path-connected parameter space \(C\):
\begin{equation}\label{eq:connectedtensorcotensor}
	\TOP(C,\FLOW(X,Y))
	\cong
	\FLOW(C\boxtimes X,Y)
	\cong
	\FLOW(X,\{C,Y\}_S).
\end{equation}
The left-hand homeomorphism is precisely \cite[Theorem~9.6]{model3}. The right-hand homeomorphism states that the adjunction is compatible with the \(\Top\)-enrichment. This statement does not appear explicitly in \cite{model3}. It is derived from the Yoneda lemma as follows. Let \(Z\in \Top\). First we observe that
\[
\begin{aligned}
	&\Top(Z,\FLOW(C\boxtimes X,Y))\cong  \prod_{c\in\pi_0(Z)} \Top(Z_c,\FLOW(C\boxtimes X,Y)), \\[2mm]
    &\Top(Z,\FLOW(X,\{C,Y\}_S)) \cong \prod_{c\in\pi_0(Z)}\Top(Z_c,\FLOW(X,\{C,Y\}_S)).
\end{aligned}
\]
Thus we can assume that \(Z\) is non-empty and path-connected. Then, using \cite[Theorems 7.8 and 7.9]{model3}, we write the sequence of natural bijections
\[
\begin{aligned}
	\Top(Z,\FLOW(C\boxtimes X,Y)) &\cong \Flow(C\boxtimes (Z\boxtimes X),Y)\\
	&\cong \Flow((C\times Z)\boxtimes X,Y)\\
	&\cong \Flow(Z\boxtimes X,\{C,Y\}_S)\\
	&\cong \Top(Z,\FLOW(X,\{C,Y\}_S)).
\end{aligned}
\]
By the Yoneda lemma, the natural map
\[
\FLOW(C\boxtimes X,Y) \cong \FLOW(X,\{C,Y\}_S)
\]
is therefore a homeomorphism. We shall also use the globe formula from \cite[Theorem~7.8]{model3}: for \(C\) path-connected and non-empty and for every \(Z\in\Top\),
\begin{equation}\label{eq:globeformula}
	C\boxtimes \Glob(Z)\cong \Glob(C\times Z).
\end{equation}

Let \(U\in\Top\) and let \(X\) be a flow.  Define
\begin{equation}\label{eq:tensor-definition}
	U\otimes X
	:=
	\coprod_{c\in\pi_0(U)} \bigl(U_c\boxtimes X\bigr),
\end{equation}
where the coproduct is taken in \(\Flow\).  If \(U=\varnothing\), this is the initial flow. Observe that
\[
(U\otimes X)^0
\cong
\coprod_{c\in\pi_0(U)} X^0,
\]
as expected: a family of maps \(U\to\FLOW(X,Y)\) may choose a different map on states \(X^0\to Y^0\) on each path-connected component of \(U\).

\begin{samepage}
\begin{proposition}\label{prop:tensor}
For any two flows \(X,Y\) and any space \(U\), there is a natural homeomorphism
\[
	\FLOW(U\otimes X,Y)
	\cong
	\TOP(U,\FLOW(X,Y)).
	\]
\end{proposition}
\end{samepage}

\begin{proof}
Using the definition \eqref{eq:tensor-definition}, the universal property of coproducts in \(\Flow\), the connected case \eqref{eq:connectedtensorcotensor}, and finally \eqref{eq:mapsfromcomponentcoproduct}, we obtain natural homeomorphisms
\[
	\begin{aligned}
		\FLOW(U\otimes X,Y)
		&\cong
		\prod_{c\in\pi_0(U)} \FLOW(U_c\boxtimes X,Y) \\
		&\cong
		\prod_{c\in\pi_0(U)} \TOP(U_c,\FLOW(X,Y)) \\
		&\cong
		\TOP\!\left(\coprod_{c\in\pi_0(U)}U_c,\FLOW(X,Y)\right) \\
		&\cong
		\TOP(U,\FLOW(X,Y)).
	\end{aligned}
	\]
All homeomorphisms are natural in \(U\), \(X\) and \(Y\).
\end{proof}

Let \(U\in\Top\) and let \(Y\) be a flow.  Define
\begin{equation}\label{eq:cotensor-definition}
	Y^U
	:=
	\prod_{c\in\pi_0(U)} \{U_c,Y\}_S,
\end{equation}
where the product is taken in \(\Flow\).  If \(U=\varnothing\), this is the terminal flow.

\begin{samepage}
\begin{proposition}\label{prop:cotensor}
For any two flows \(X,Y\) and any space \(U\), there is a natural homeomorphism
\[
	\FLOW(X,Y^U)
	\cong
	\TOP(U,\FLOW(X,Y)).
	\]
\end{proposition}
\end{samepage}


\begin{proof}
Using the definition \eqref{eq:cotensor-definition}, the universal property of products in \(\Flow\), the connected case \eqref{eq:connectedtensorcotensor}, and \eqref{eq:mapsfromcomponentcoproduct}, we obtain natural homeomorphisms
\[
	\begin{aligned}
		\FLOW(X,Y^U)
		&\cong
		\prod_{c\in\pi_0(U)} \FLOW(X,\{U_c,Y\}_S) \\
		&\cong
		\prod_{c\in\pi_0(U)} \TOP(U_c,\FLOW(X,Y)) \\
		&\cong
		\TOP\!\left(\coprod_{c\in\pi_0(U)}U_c,\FLOW(X,Y)\right) \\
		&\cong
		\TOP(U,\FLOW(X,Y)).
	\end{aligned}
	\]
Again, the construction is natural in \(U\), \(X\) and \(Y\).
\end{proof}

Combining Propositions \ref{prop:tensor} and \ref{prop:cotensor}, the standard \(\Top\)-enriched category of flows is tensored and cotensored over \(\Top\) in the sense of \cite[Section~6.5]{Borceux2}.  Explicitly, for \(U\in\Top\) and flows \(X,Y\), the tensor and cotensor are
\[
U\otimes X=\coprod_{c\in\pi_0(U)}(U_c\boxtimes X),
\qquad
Y^U=\prod_{c\in\pi_0(U)}\{U_c,Y\}_S,
\]
and they satisfy the enriched adjunction homeomorphisms
\[
\FLOW(U\otimes X,Y)
\cong
\TOP(U,\FLOW(X,Y))
\cong
\FLOW(X,Y^U).
\]
When \(U\) is path-connected and non-empty, the formulas reduce to the classical \(U\boxtimes X\) and \(\{U,Y\}_S\).

Since \(\Flow\) is complete and cocomplete, it is therefore topologically bicomplete in the sense of Barthel--Riehl \cite{Barthel-Riel}. Since \(\Flow\) is locally presentable by \cite[Theorem~6.11]{Moore1} and topologically bicomplete, it satisfies the monomorphism hypothesis of \cite[Definition~5.16]{Barthel-Riel} by \cite[Remark~5.20]{Barthel-Riel}.

We now write H with a capital letter for the Hurewicz model structure on \(\Flow\) obtained from \cite[Corollary~5.23]{Barthel-Riel}.  Thus the weak equivalences are the homotopy equivalences defined by the topological enrichment, the fibrations are the maps that have the right lifting property against all cylinder inclusions of the form \(i_0: Z \to I\otimes Z \), and the cofibrations are the maps that have the left lifting property against all H-fibrations that are also homotopy equivalences. The H-cofibrations have the left lifting property against all maps of the form \(p_0: Z^{I} \to Z\). This terminology avoids any conflict with the h-model structure of flows of \cite{QHMmodel}: in both model structures, all objects are fibrant; however, there exist flows which are not h-cofibrant by \cite[Proposition~7.9]{QHMmodel}, whereas all flows are H-cofibrant.

\section{Comparison between the q- and the H-model structures}
\label{sec:comparison-q-H}

We first compare the homotopies in the two model structures.

\begin{samepage}
\begin{lem}\label{lem:homotopic-flow-maps}
Let \(f,g:X\to Y\) be two maps of flows.  If \(f\) and \(g\) are homotopic in the H-model structure, then \(f^0=g^0:X^0\to Y^0\), and the maps
\[
	\PP f,\PP g:\PP X\longrightarrow \PP Y
	\]
are homotopic maps of spaces.  More precisely, for every \(\alpha,\beta\in X^0\), the restrictions
\[
	\PP_{\alpha,\beta}f,\PP_{\alpha,\beta}g:
	\PP_{\alpha,\beta}X\longrightarrow \PP_{f(\alpha),f(\beta)}Y
	\]
are homotopic.
\end{lem}
\end{samepage}

\begin{proof}
A homotopy from \(f\) to \(g\) is a map
\[
	H:I\otimes X\longrightarrow Y,
	\qquad I=[0,1],
	\]
whose restrictions along the two endpoint maps \(X\cong \{0\}\otimes X\to I\otimes X\) and \(X\cong \{1\}\otimes X\to I\otimes X\) are \(f\) and \(g\).  Since \(I\) is path-connected, our formula for the tensor gives
\[
	I\otimes X=I\boxtimes X
	\qquad\hbox{and}\qquad
	(I\otimes X)^0=X^0.
	\]
The two endpoint maps are the identity on the set of states.  Hence the two restrictions of \(H\) have the same map on states, i.e. \(f^0=g^0\). By the tensor adjunction, \(H\) is equivalently a continuous map
\[
	\widehat H:I\longrightarrow \FLOW(X,Y)
	\]
from \(f\) to \(g\).  The composite
\[
	I\times \PP X\xrightarrow{\ \widehat H\times \operatorname{Id}\ }
	\FLOW(X,Y)\times \PP X\xrightarrow{\operatorname{ev}} \PP Y
	\]
is a homotopy from \(\PP f\) to \(\PP g\). Since \(f^0=g^0\), it restricts to a homotopy \(H_{\alpha,\beta}\) for every \(\alpha,\beta\in X^0\):
\[
\begin{tikzcd}[column sep=huge,row sep=large]
\PP_{\alpha,\beta}X \arrow[d,"i_0"'] \arrow[dr,"\PP_{\alpha,\beta}f"] & \\
I\times\PP_{\alpha,\beta}X \arrow[r,"H_{\alpha,\beta}"'] & \PP_{f(\alpha),f(\beta)}Y \\
\PP_{\alpha,\beta}X \arrow[u,"i_1"] \arrow[ur,"\PP_{\alpha,\beta}g"'] &
\end{tikzcd}
\]
\end{proof}

\begin{samepage}
\begin{proposition}\label{prop:Hweak-equivalence-state-path}
Let \(f:X\to Y\) be a weak equivalence of the H-model structure of flows.  Then \(f^0:X^0\to Y^0\) is a bijection, and
\[
	\PP f:\PP X\longrightarrow \PP Y
	\]
is a homotopy equivalence of spaces.  In fact, for every \(\alpha,\beta\in X^0\), the induced map
\[
	\PP_{\alpha,\beta}X\longrightarrow \PP_{f(\alpha),f(\beta)}Y
	\]
is a homotopy equivalence.
\end{proposition}
\end{samepage}

In the language of \cite{QHMmodel}, every weak equivalence of the H-model structure of flows is a weak equivalence of the h-model structure. The converse is false, since a weak equivalence of the h-model structure is not necessarily invertible up to homotopy.

\begin{proof}
By definition of the weak equivalences in the H-model structure, \(f\) is a homotopy equivalence in the topologically enriched category \(\Flow\).  Choose a homotopy inverse \(g:Y\to X\) and homotopies
\[
	gf\simeq \operatorname{Id}_X,
	\qquad
	fg\simeq \operatorname{Id}_Y.
	\]
Applying Lemma \ref{lem:homotopic-flow-maps} to these homotopies gives
\[
	g^0f^0=\operatorname{Id}_{X^0},
	\qquad
	f^0g^0=\operatorname{Id}_{Y^0},
	\]
so \(f^0\) is a bijection.  The same lemma gives homotopies of spaces
\[
	\PP g\,\PP f\simeq \operatorname{Id}_{\PP X},
	\qquad
	\PP f\,\PP g\simeq \operatorname{Id}_{\PP Y}.
	\]
Thus \(\PP f\) is a homotopy equivalence of spaces, with homotopy inverse \(\PP g\). Finally, because the maps on states of \(gf\) and \(\operatorname{Id}_X\) are equal, and likewise for \(fg\) and \(\operatorname{Id}_Y\), the same argument restricts to each summand \(\PP_{\alpha,\beta}X\).
\end{proof}

\begin{samepage}
\begin{lem}\label{lem:H-fibration-path-spaces}
Let \(p:E\to B\) be a fibration of the H-model structure of flows.  Then for every \(\alpha,\beta\in E^0\), the induced map
\[
	\PP_{\alpha,\beta}p:
	\PP_{\alpha,\beta}E\longrightarrow \PP_{p(\alpha),p(\beta)}B
	\]
is a Hurewicz fibration of spaces.
\end{lem}
\end{samepage}

\begin{proof}
Let \(Z\in\Top\).  A lifting problem in spaces
\[
	\begin{tikzcd}
		Z\times\{0\} \arrow[r] \arrow[d] & \PP_{\alpha,\beta}E \arrow[d] \\
		Z\times I \arrow[r] \arrow[ru,dashed] & \PP_{p(\alpha),p(\beta)}B
	\end{tikzcd}
	\]
is equivalently a lifting problem of flows
\[
	\begin{tikzcd}
		\Glob(Z) \arrow[r] \arrow[d] & E \arrow[d,"p"] \\
		\Glob(Z\times I) \arrow[r] \arrow[ru,dashed] & B
	\end{tikzcd}
	\]
Here the left vertical map is induced by \(Z\times\{0\}\subset Z\times I\), and the maps on states are determined by \(\alpha,\beta\) and \(p(\alpha),p(\beta)\).  By the globe formula \eqref{eq:globeformula}, this left vertical map is the cylinder inclusion
\[
	\Glob(Z)\longrightarrow I\otimes \Glob(Z).
	\]
Since \(p\) is an H-fibration, it has the right lifting property with respect to all such cylinder inclusions.  The resulting lift is precisely a lift in the original topological lifting problem.  Hence \(\PP_{\alpha,\beta}p\) is a Hurewicz fibration.
\end{proof}

\begin{samepage}
\begin{thm}\label{thm:q-cof-H-cof}
Every q-cofibration of flows is an H-cofibration of flows.
\end{thm}
\end{samepage}

\begin{proof}
Since the q-cofibrations are the retracts of relative \((I^{\mathrm{gl}}\cup\{C,R\})\)-cell complexes and since the class of H-cofibrations is closed under pushouts, transfinite compositions and retracts, it suffices to prove that every element of \(I^{\mathrm{gl}}\cup\{C,R\}\) is an H-cofibration.

Let \(p:E\to B\) be a trivial H-fibration.  By Proposition \ref{prop:Hweak-equivalence-state-path}, the map \(p^0:E^0\to B^0\) is a bijection. Therefore \(C\) has the left lifting property with respect to \(p\): a map \(\{0\}\to B\) is just a state of \(B\), and it has a unique preimage in \(E^0\).  The fact that this map is a bijection also implies that \(R\) has the left lifting property with respect to \(p\): if two states of \(E\) have the same image in \(B\), then they are equal. It remains to treat the globular generators of \(I^{\mathrm{gl}}\).  Let \(j:\Sp^{n-1}\to \Di^n\) be the standard inclusion.  A lifting problem
\[
	\begin{tikzcd}
		\Glob(\Sp^{n-1}) \arrow[r] \arrow[d] & E \arrow[d,"p"] \\
		\Glob(\Di^n) \arrow[r] \arrow[ru,dashed] & B
	\end{tikzcd}
	\]
fixes two states \(\alpha,\beta\in E^0\) and is exactly the same thing as a lifting problem in \(\Top\)
\[
	\begin{tikzcd}
		\Sp^{n-1} \arrow[r] \arrow[d,"j"'] & \PP_{\alpha,\beta}E \arrow[d,"\PP_{\alpha,\beta}p"] \\
		\Di^n \arrow[r] \arrow[ru,dashed] & \PP_{p(\alpha),p(\beta)}B
	\end{tikzcd}
	\]
By Lemma \ref{lem:H-fibration-path-spaces}, the right vertical map is a Hurewicz fibration of spaces.  By Proposition \ref{prop:Hweak-equivalence-state-path}, it is also a homotopy equivalence.  Thus it is a trivial fibration in the Hurewicz model structure of \(\Top\). The inclusion \(\Sp^{n-1}\subset \Di^n\) is a closed NDR pair, hence a cofibration in that Hurewicz model structure.  Therefore the displayed lifting problem has a solution.  This proves that every globular generator is an H-cofibration, and the theorem follows.
\end{proof}

\begin{samepage}
\begin{proposition}\label{prop:q-whitehead-to-H}
Let \(f:X\to Y\) be a weak equivalence of the q-model structure of flows.  If \(X\) and \(Y\) are q-cofibrant, then \(f\) is a weak equivalence of the H-model structure.  Equivalently, \(f\) is a homotopy equivalence for the topological enrichment of \(\Flow\).
\end{proposition}
\end{samepage}

\begin{proof}
All flows are fibrant in the q-model structure.  Hence \(X\) and \(Y\) are cofibrant-fibrant objects of the q-model structure.  By the usual Whitehead theorem for model categories, a weak equivalence between cofibrant-fibrant objects admits a homotopy inverse, with homotopies computed using any cylinder object of the q-model structure. By \cite[Corollary~7.11]{model3}, and since \(I=[0,1]\) is path-connected, two maps of flows \(f,g:X\to Y\) are homotopic for the q-model structure if there exists a map
\[
	H:I\boxtimes X = I\otimes X \longrightarrow Y
	\]
such that the composite
\[
	X\amalg X\longrightarrow I\otimes X\xlongrightarrow{H} Y,
	\qquad I=[0,1],
	\]
is equal to \((f,g)\). Therefore the homotopy inverse supplied by the Whitehead criterion is also an inverse up to homotopy in the H-model structure.  This is precisely to say that \(f\) is a weak equivalence for the H-model structure.
\end{proof}

\section{Pushout along a q-cofibrant replacement}
\label{sec:pushout-along-q-replacement}

We now generalize Theorem~\ref{thm:canonical} to any q-cofibrant replacement of any order-reflecting inclusion of posets \(P\subset Q\).

\begin{samepage}
\begin{proposition} \label{prop:useful-lemma}
Let \(\mathcal M\) be a model category. Consider a commutative diagram of solid arrows
\[
	\begin{tikzcd}
		&&\\
		A \arrow[d,rightarrowtail] \arrow[dashed,r,"\ell"] \arrow[rr,bend left=30pt]  & B \arrow[d]\ \arrow[r,twoheadrightarrow,"\simeq"] & C \arrow[d] \\
		A' \arrow[dashed,r,"\ell'"] \arrow[rr,bend right=30pt] & B' \arrow[twoheadrightarrow,r,"\simeq"] & C'\\
		&&
	\end{tikzcd}
	\]
such that the maps \(B\to C\) and \(B'\to C'\) are trivial fibrations and such that the map \(A\to A'\) is a cofibration between cofibrant objects. Then there exist maps \(\ell:A\to B\) and \(\ell':A'\to B'\) making the diagram commutative.
\end{proposition}
\end{samepage}

\begin{proof}
It is \cite[Proposition~2.5]{MultipointedSubdivision}.
\end{proof}

\begin{samepage}
\begin{proposition} \label{prop:q-2-h-bis}
Let \(f:X\to Y\) be a q-cofibration of flows. Then for every \(u,v\in X^0\), the induced map
\[
	\PP_{u,v} X \xlongrightarrow{f}\PP_{f(u),f(v)}Y
	\]
is an h-cofibration of spaces.
\end{proposition}
\end{samepage}

When \(X\) and \(Y\) are q-cofibrant, the result follows from \cite[Theorem~5.7]{leftproperflow}. The following proof does not require this assumption.

\begin{proof}
Every q-cofibration is a retract of a relative \(\bigl(I^{\mathrm{gl}}\cup\{C,R\}\bigr)\)-cell complex. After refining the cellular presentation, we may suppose that one generating q-cofibration is attached at each successor stage.

For a pushout of a map of \(I^{\mathrm{gl}}\), the induced maps on spaces of execution paths are h-cofibrations by \cite[Theorem~5.4(1)]{leftproperflow}. A pushout of \(C\) induces a homeomorphism on spaces of execution paths. Finally, consider a pushout of \(R\), and denote its map on states by \(q\). If the two chosen states are already equal, the pushout is an isomorphism. Otherwise, the description by reduced words of this pushout gives, for every pair of states \(a,b\), a homeomorphism
\[
	\PP_{q(a),q(b)}Y
	\cong
	\PP_{a,b}X\amalg Z_{a,b}
	\]
for some space \(Z_{a,b}\), under which
\[
	\PP_{a,b}X\longrightarrow
	\PP_{q(a),q(b)}Y
	\]
is the inclusion of the first coproduct summand. Since every topological space is h-cofibrant, this map is an h-cofibration.

Consider now a cellular tower
\[
	X=X_0\longrightarrow X_1\longrightarrow\cdots\longrightarrow
	X_\lambda=Y,
	\]
and let \(u_\alpha,v_\alpha\) be the images in \(X_\alpha^0\) of fixed states \(u,v\in X^0\). The preceding discussion shows that every successor map
\[
	\PP_{u_\alpha,v_\alpha}X_\alpha
	\longrightarrow
	\PP_{u_{\alpha+1},v_{\alpha+1}}X_{\alpha+1}
	\]
is an h-cofibration. It also shows that every map
\[
	\PP X_\alpha\longrightarrow\PP X_{\alpha+1}
	\]
is an h-cofibration, and hence a relative-\(T_1\) inclusion. At every nonzero limit ordinal \(\beta\leq\lambda\), the homeomorphism of Theorem~\ref{thm:path-almost-accessible} restricts to a homeomorphism
\[
	\injlim_{\alpha<\beta}
	\PP_{u_\alpha,v_\alpha}X_\alpha
	\cong
	\PP_{u_\beta,v_\beta}X_\beta.
	\]
Consequently,
\[
	\PP_{u,v}X\longrightarrow
	\PP_{f(u),f(v)}Y
	\]
is a transfinite composite of h-cofibrations and is therefore an h-cofibration. The result follows since h-cofibrations are closed under retracts.
\end{proof}

\begin{samepage}
\begin{definition} \label{def:q-cofibrant-replacement-inclusion-poset}
A \textit{q-cofibrant replacement} of a map of flows \(f:X\to Y\) is a commutative square
\[\begin{tikzcd}[row sep=2.5em] X^{\mathrm{cof}} \arrow[d,rightarrowtail,"f^{\mathrm{cof}}"'] \arrow[r,twoheadrightarrow,"\sim"]& X \arrow[d,"f"]\\ Y^{\mathrm{cof}} \arrow[r,twoheadrightarrow,"\sim"]&Y \end{tikzcd}\]
in which \(X^{\mathrm{cof}}\) and \(Y^{\mathrm{cof}}\) are q-cofibrant, the horizontal maps are trivial q-fibrations, and \(f^{\mathrm{cof}}\) is a q-cofibration.
\end{definition}
\end{samepage}

\begin{samepage}
\begin{thm} \label{thm:r}
Let \(P\subset Q\) be an order-reflecting inclusion of posets, and let
\[
	i^{\mathrm{cof}}:P^\mathrm{cof}\longrightarrow Q^\mathrm{cof}
	\]
be a q-cofibrant replacement.  Let
\[
	\begin{tikzcd}[row sep=2.5em]
		P^\mathrm{cof} \ar[r,"\phi"] \ar[d,"i^{\mathrm{cof}}"',rightarrowtail] & X \ar[d,"f",rightarrowtail] \\
		Q^\mathrm{cof} \ar[r,"\phi'"] & Y \cocartesian
	\end{tikzcd}
	\]
be a pushout of flows.  Then for all states \(u,v\in X^0\), the induced map
\[
	\PP_{u,v}X\longrightarrow \PP_{f(u),f(v)}Y
	\]
is a trivial h-cofibration of topological spaces.
\end{thm}
\end{samepage}

We apply the gluing lemma \cite[Lemma~8.12]{MR2001d:55012} in the H-model structure of Section~\ref{sec:H-model-structure}. All flows are H-cofibrant and H-fibrant, so this model structure is proper.

\begin{proof}
Consider the commutative diagram of flows
\[
	\begin{tikzcd}
		&&\\
		\Ch(P) \arrow[d,rightarrowtail] \arrow[dashed,r,"\ell"] \arrow[rr,bend left=30pt]  & P^\mathrm{cof} \arrow[d]\ \arrow[r,twoheadrightarrow,"\simeq"] & P \arrow[d] \\
		\Ch(Q) \arrow[dashed,r,"\ell'"] \arrow[rr,bend right=30pt] & Q^\mathrm{cof} \arrow[twoheadrightarrow,r,"\simeq"] & Q\\
		&&
	\end{tikzcd}
	\]
The existence of \(\ell\) and \(\ell'\) is a consequence of Proposition~\ref{prop:useful-lemma}: they are weak equivalences of flows by the two-out-of-three property. Consider the flow \(Y^{\mathrm{ch}}\) defined by the pushout diagram
\[
	\begin{tikzcd}[column sep=5em]
		\Ch(P) \arrow[r,"\phi\ell"] \arrow[d,"\Ch(i)"',rightarrowtail]& X\arrow[d,"g",rightarrowtail]\\
		\Ch(Q) \arrow[r,"j"'] & \cocartesian Y^{\mathrm{ch}}.
	\end{tikzcd}
	\]
Let \(k:Y^{\mathrm{ch}}\to Y\) be the canonical map, so that \(kg=f\) and \(kj=\phi'\ell'\). We obtain the commutative diagram
\[
	\begin{tikzcd}
		\Ch(P) \arrow[dd,rightarrowtail]\arrow[rd,"\simeq"]\arrow[rr] && X \arrow[dd,rightarrowtail] \arrow[rd,equal]\\
		& P^\mathrm{cof} \arrow[rr,crossing over] & & X \arrow[dd,rightarrowtail]\\
		\Ch(Q) \arrow[rd,"\simeq"]\arrow[rr] && Y^{\mathrm{ch}} \arrow[uull,phantom,"\ulcorner"{font=\Large},pos=0] \arrow[rd,dashed,"k",shorten >=0.6em] \\
		& Q^\mathrm{cof} \arrow[rr] \arrow[uu,leftarrowtail,crossing over] && Y \arrow[uull,phantom,"\ulcorner"{font=\Large},pos=0]
	\end{tikzcd}
	\]
The flows \(\Ch(P)\) and \(\Ch(Q)\) are q-cofibrant by Proposition~\ref{prop:chain-replacement}, and \(P^\mathrm{cof}\) and \(Q^\mathrm{cof}\) are q-cofibrant by hypothesis. Thus the maps \(\Ch(P)\to P^\mathrm{cof}\) and  \(\Ch(Q)\to Q^\mathrm{cof}\) are weak equivalences of the H-model structure of flows by Proposition~\ref{prop:q-whitehead-to-H}. The vertical maps \(\Ch(i):\Ch(P)\to\Ch(Q)\) and \(i^{\mathrm{cof}}:P^\mathrm{cof}\to Q^\mathrm{cof}\)  are H-cofibrations of flows by Theorem~\ref{thm:q-cof-H-cof}. By the gluing lemma \cite[Lemma~8.12]{MR2001d:55012} applied in the H-model structure of flows, the map \(Y^{\mathrm{ch}}\to Y\) is a weak equivalence of the H-model structure of flows. For \(u,v\in X^0\), consider the commutative triangle
\[
\begin{tikzcd}[column sep=large,row sep=large]
\PP_{u,v}X \arrow[r,rightarrowtail,"\simeq"] \arrow[dr] & \PP_{g(u),g(v)}Y^{\mathrm{ch}} \arrow[d,"\simeq"] \\
& \PP_{f(u),f(v)}Y.
\end{tikzcd}
\]
The horizontal map is a trivial h-cofibration by Theorem~\ref{thm:canonical}. The vertical map is a homotopy equivalence by Proposition~\ref{prop:Hweak-equivalence-state-path}. Thus their composite is a homotopy equivalence. Moreover, \(f\) is a q-cofibration, as a pushout of a q-cofibration. Proposition~\ref{prop:q-2-h-bis} shows that the same composite is an h-cofibration. It is therefore a trivial h-cofibration.
\end{proof}

\begin{samepage}
\begin{cor} \label{cor:final}
Let \(f:X\to Y\) be a retract of a transfinite composition of pushouts of q-cofibrant replacements of order-reflecting inclusions of posets. Then, for any pair of states \(u,v\in X^0\), \(f\) induces a trivial h-cofibration of spaces
\[
	\begin{tikzcd}
		\PP_{u,v} X \arrow[r,"\simeq"] & \PP_{f(u),f(v)} Y.
	\end{tikzcd}
	\]
In particular, it is a homotopy equivalence.
\end{cor}
\end{samepage}

\begin{proof}
We first consider a transfinite composition
\[
	X_0\longrightarrow X_1\longrightarrow\cdots\longrightarrow
	X_\lambda
	\]
in which, for every \(\alpha<\lambda\), the map
\[
	j_\alpha: X_\alpha\longrightarrow X_{\alpha+1}
	\]
is a pushout of a \(q\)-cofibrant replacement of an order-reflecting inclusion of posets.

The map on states of every such q-cofibrant replacement is injective. Indeed, in the commutative square of Definition~\ref{def:q-cofibrant-replacement-inclusion-poset}, the two trivial q-fibrations induce bijections on the sets of states, and hence the map on states of the q-cofibrant replacement is conjugate, by these bijections, to the given inclusion of posets. Since pushouts and transfinite compositions of injections are injections in \(\Set\), all the maps
\[
	j_\alpha^0: X_\alpha^0\longrightarrow X_{\alpha+1}^0
	\]
are injective.

Fix \(u,v\in X_0^0\), and denote by \(u_\alpha,v_\alpha\) their images in \(X_\alpha^0\). By Theorem~\ref{thm:r}, for every \(\alpha<\lambda\), the map
\[
	\PP_{u_\alpha,v_\alpha}X_\alpha
	\longrightarrow
	\PP_{u_{\alpha+1},v_{\alpha+1}}X_{\alpha+1}
	\]
is a trivial h-cofibration.

We now verify the behavior at limit ordinals. For every \(\alpha<\lambda\), the map
\[
	\PP X_\alpha\longrightarrow \PP X_{\alpha+1}
	\]
factors as
\[
	\coprod_{(a,b)\in X_\alpha^0\times X_\alpha^0}
	\PP_{a,b}X_\alpha
	\longrightarrow
	\coprod_{(a,b)\in X_\alpha^0\times X_\alpha^0}
	\PP_{j_\alpha(a),j_\alpha(b)}X_{\alpha+1}
	\longrightarrow
	\PP X_{\alpha+1}.
	\]
The first map is a coproduct of trivial h-cofibrations by Theorem~\ref{thm:r}. The second map is the inclusion of a coproduct summand, because \(j_\alpha^0\) is injective. Since every topological space is h-cofibrant, this second map is an h-cofibration. Consequently,
\[
	\PP X_\alpha\longrightarrow \PP X_{\alpha+1}
	\]
is an h-cofibration and therefore a relative-\(T_1\) inclusion.

Let \(\beta\leq\lambda\) be a nonzero limit ordinal. Theorem~\ref{thm:path-almost-accessible} gives a homeomorphism
\[
	\injlim_{\alpha<\beta}\PP X_\alpha
	\cong
	\PP X_\beta.
	\]
Since all maps on states in the tower are injective, this homeomorphism restricts to a homeomorphism
\[
	\injlim_{\alpha<\beta}\PP_{u_\alpha,v_\alpha}X_\alpha
	\cong
	\PP_{u_\beta,v_\beta}X_\beta.
	\]
Thus the tower of path spaces
\[
	\bigl(
	\PP_{u_\alpha,v_\alpha}X_\alpha
	\bigr)_{\alpha\leq\lambda}
	\]
is continuous at every limit ordinal, and all its successor maps are trivial h-cofibrations. Since trivial h-cofibrations are closed under transfinite composition, the induced map
\[
	\PP_{u,v}X_0
	\longrightarrow
	\PP_{u_\lambda,v_\lambda}X_\lambda
	\]
is a trivial h-cofibration.

Finally, let \(f:X\to Y\) be a retract of such a transfinite composition \(g:A\to B\). Write the retraction as the commutative diagram
\[
\begin{tikzcd}[column sep=large,row sep=large]
X \arrow[r,"r_0"] \arrow[d,"f"'] & A \arrow[r,"s_0"] \arrow[d,"g"] & X \arrow[d,"f"] \\
Y \arrow[r,"r_1"'] & B \arrow[r,"s_1"'] & Y,
\end{tikzcd}
\qquad s_0r_0=\operatorname{Id}_X,\quad s_1r_1=\operatorname{Id}_Y.
\]
For every \(u,v\in X^0\), the map
\[
	\PP_{u,v}X
	\longrightarrow
	\PP_{f(u),f(v)}Y
	\]
is then a retract, in the arrow category of topological spaces, of
\[
	\PP_{r_0(u),r_0(v)}A
	\longrightarrow
	\PP_{g(r_0(u)),g(r_0(v))}B.
	\]
The latter is a trivial h-cofibration by the preceding argument. Since trivial h-cofibrations are closed under retracts, the map induced by \(f\) is a trivial h-cofibration.
\end{proof}


\end{document}